\documentclass[11pt,a4paper,oneside]{article}
\usepackage{setspace}
\usepackage{fancyhdr}
\usepackage{tabularx}
\usepackage[a4paper,left=1cm,right=1cm,top=1cm,bottom=2cm]{geometry}

\usepackage[utf8]{inputenc}
\usepackage{graphicx}
\usepackage{amssymb}
\usepackage{amsmath}
\usepackage{amsthm}
\usepackage{subcaption}
\usepackage[affil-it]{authblk}
\usepackage[natbib=true,style=numeric,sorting=none,citestyle=numeric-comp,style=nature]{biblatex}
\usepackage{comment}

\bibliography{sample}

\newtheorem{theorem}{Theorem}
\newtheorem{lemma}{Lemma}
\newtheorem{corollary}{Corollary}
\newtheorem*{remark}{Remark}

\title{\textbf{The Spherical Kapitza-Whitney Pendulum}}

\author{\textbf{Ivan Polekhin}%
  \thanks{Electronic address: \texttt{ivanpolekhin@mi-ras.ru}}}
\affil{Steklov Mathematical Institute of Russian Academy of Sciences,\\ Moscow, Russia}

\begin{document}

\maketitle

\section*{Abstract}

In this paper we study the global dynamics of the inverted spherical pendulum with a vertically vibrating suspension point in the presence of an external horizontal periodic force field. We do not assume that this force field is weak or rapidly oscillating. We prove that there always exists a non-falling periodic solution, i.e., there exists an initial condition such that the rod of the pendulum never becomes horizontal along the corresponding solution. We also show numerically that there exists an asymptotically stable non-falling solution for a wide range of parameters of the system.
\vspace{10pt}\\
\noindent\textbf{Keywords: } forced oscillations, Kapitza pendulum, Whitney pendulum, stabilization, vibrations.

\section{Introduction}

The simple planar pendulum with a vibrating pivot point has been one of the most well-studied nonlinear systems in the 20th century \cite{stephenson1908xx,bogolyubov1950perturbation,kapitsa1951pendulum,kapitsa1951dynamic} and still remains an object of mathematical research (see, for instance, the most recent papers \cite{artstein2021pendulum,araujo2021parametric,belyaev2021classical,cabral2021parametric}). A comprehensive overview of the works related to the topic and the history of the problem can be found in \cite{butikov2001dynamic,samoilenko1994nn}. In contrast, far less attention has been paid to the spherical pendulum with a vibrating pivot point, a natural generalization of the planar system. The work \cite{markeyev1999dynamics} by A.P. Markeev, to the best of our knowledge, is among the first papers to study this system. In particular, in this paper the existence of solutions which are impossible in the planar case is proved. To be more precise, it is shown that there exist solutions close to the conical motions of the pendulum for which the rod of the pendulum makes a constant angle with the vertical. Recently, these results have been developed by R.E. Grundy \cite{grundy2021two}. Another area of research related to the dynamics of the spherical pendulum stems from the problem of stabilization of an arbitrary position of the rod. Here the works by A.G. Petrov \textit{et al.} should be mentioned \cite{petrov2005equations,bulanchuk2010controlling}.

Our results belong rather to the area initially developed by A.P. Markeev: we also prove the existence of certain solutions. However, the setting of our problem significantly generalizes the original system: we consider the spherical pendulum with a vibrating suspension point under the action of a horizontal periodic force and prove the existence of periodic solutions in this system. Note that we do not assume that the horizontal force is weak or rapidly oscillating. In our considerations we mix analytical and topological techniques. We use a topological method developed in \cite{srzednicki1994periodic} to prove the existence of a periodic solution for some modified system and then use the classical results on averaging to prove that this solution remains in the original system. Our system can be considered as a natural combination of the classical Kapitza pendulum with the so-called Whitney pendulum, an inverted pendulum moving in the presence of an additional periodic horizontal force field \cite{bolotin2015calculus}. The current work is a development of an earlier paper \cite{polekhin2020method}, where the planar Kapitza-Whitney pendulum is studied.

The structure of the paper is as follows. First, we derive the equations of motion of the considered system and introduce the modified system, which will be used in the proof of our main result, presented in the successive section. We also obtain some numerical results concerning the existence of stable periodic solutions in the system. All technical lemmas are presented separately in the last section before the conclusion.

\section{Equations of motion}

Let us consider the classical spherical pendulum, a mass point moving on the surface of a sphere in a gravitational field. We additionally assume that
\begin{enumerate}
    \item There is a $T$-periodic horizontal force $P(t)$ acting on the point,
    \item The sphere moves vertically and the law of motion $h(t)$ is as follows
    $$h(t) = a \varepsilon \sin \frac{2\pi}{T}\frac{t}{\varepsilon}$$
    \item There is a force of viscous friction acting on the point.
\end{enumerate}

In the standard spherical coordinates, the position of the point is defined by the equations (Fig.~1)
\begin{align}
\begin{split}
    & x = \cos\theta\cos\varphi\\
    & y = h(t) + \cos\theta\sin\varphi\\
    & z = \sin\theta
\end{split}
\end{align}

Here and below we assume that the sphere has unit radius and the force of gravity is parallel to $Oy$-axis. It will allow us to avoid the degeneracies of the spherical coordinates in our considerations. We also assume that the mass of the point and the acceleration of gravity equal unity. These assumptions do not lead to any loss of generality.

\begin{figure}[h!]
  \centering
  \includegraphics[width=0.66\linewidth]{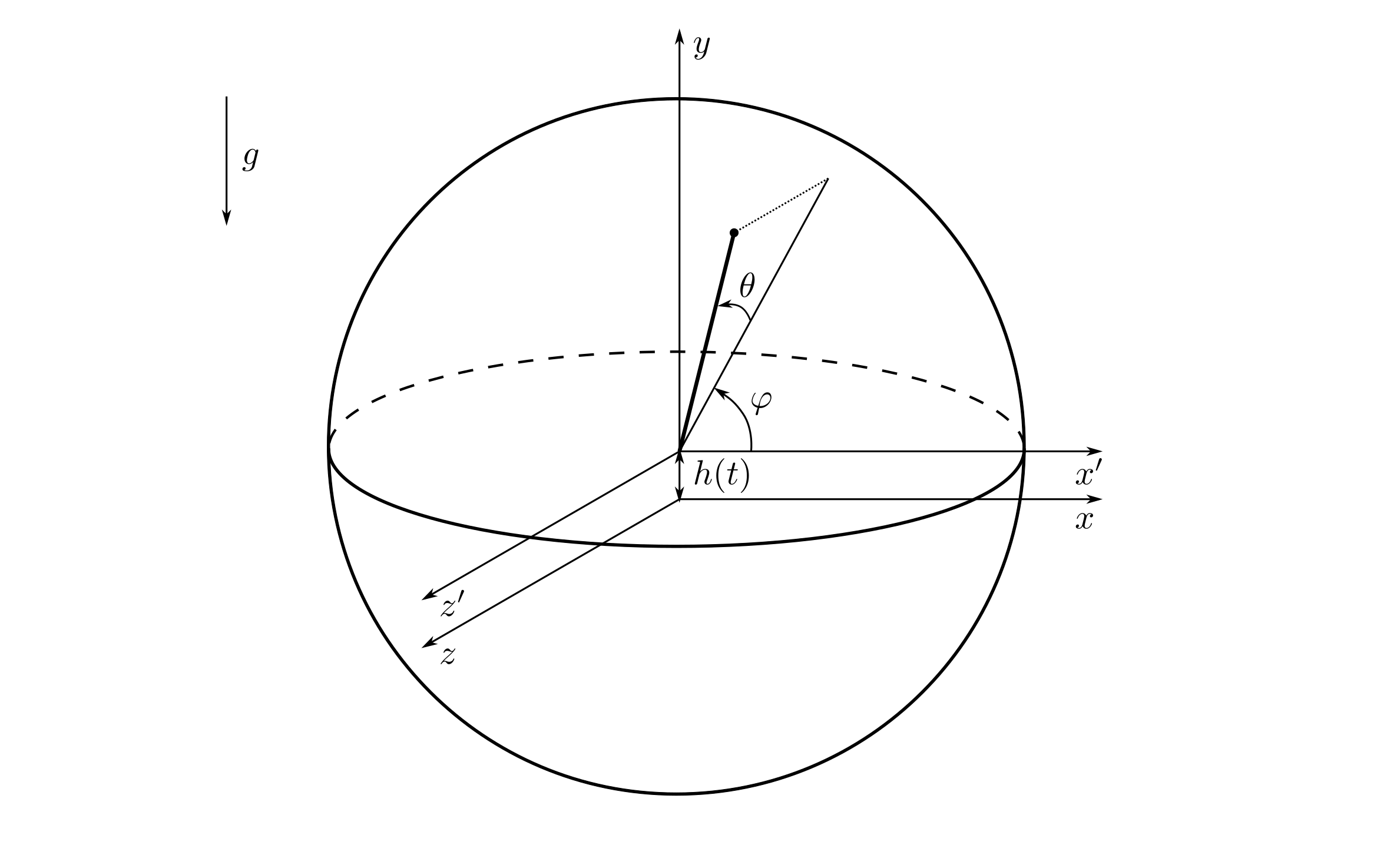}
  \caption{An inverted spherical pendulum with a vibrating pivot point.}
  \label{fig:sfig2}
\end{figure}

The kinetic and potential energies of the system can be expressed as usual
\begin{align}
    \begin{split}
        &T = \frac{1}{2}\left( \dot\theta^2 + \dot\varphi^2 \cos^2\theta - 2 \dot h \dot \theta \sin\theta \sin\varphi + 2 \dot h \dot \varphi \cos\theta \cos \varphi \right),\\
        &\Pi = \cos\theta\sin\varphi.
    \end{split}
\end{align}

We will consider the standard model for the force of viscous friction in which the resistance force $F$ is proportional to the relative velocity $v$ of the mass point
\begin{align}
    \begin{split}
        F = -\mu v = -\mu ( (-\dot\theta \sin\theta\cos\varphi - \dot\varphi\cos\theta\sin\varphi )e_x + (-\dot\theta\sin\theta\sin\varphi + \dot\varphi \cos\theta\cos\varphi)e_y + \dot\theta\cos\theta\, e_z ).
    \end{split}
\end{align}
Let $p_\theta$ and $p_\varphi$ be the conjugate momenta
\begin{align}
\label{eq4}
    p_\theta = \frac{\partial T}{\partial \dot \theta}, \quad p_\varphi = \frac{\partial T}{\partial \dot \varphi}.
\end{align}

The equations of motion can be considered as the system of four ordinary differential equations
\begin{align}
    \begin{split}
    \label{eq5}
        &\dot \theta = p_\theta + \dot h \sin\theta \sin\varphi,\\
        &\dot \varphi = (p_\varphi - \dot h \cos\theta\cos\varphi)/\cos^2\theta,\\
        &\dot p_\theta = - \dot\varphi^2 \cos\theta\sin\theta - \dot h \dot\theta \cos\theta\sin\varphi - \dot h\dot\varphi \sin\theta\cos\varphi - \mu\dot\theta + \sin\theta\sin\varphi - P_x \sin\theta\cos\varphi + P_z \cos\theta,\\
        &\dot p_\varphi = -\dot h \dot\theta \sin\theta\cos\varphi - \dot h\dot\varphi \cos\theta\sin\varphi - \cos\theta\cos\varphi - \mu \dot\varphi \cos^2\theta - P_x \cos\theta \sin\varphi.
    \end{split}
\end{align}

This system is obtained from the Lagrange equations after the substitution \eqref{eq4}. Here and below we assume that $P_x$ and $P_z$ are $T$-periodic $C^\infty$-functions.  Equations \eqref{eq5} are correctly defined for all values $\varphi$ and $\theta$ where $\cos\theta \ne 0$.

\begin{figure}[h!]
  \centering
  \includegraphics[width=0.66\linewidth]{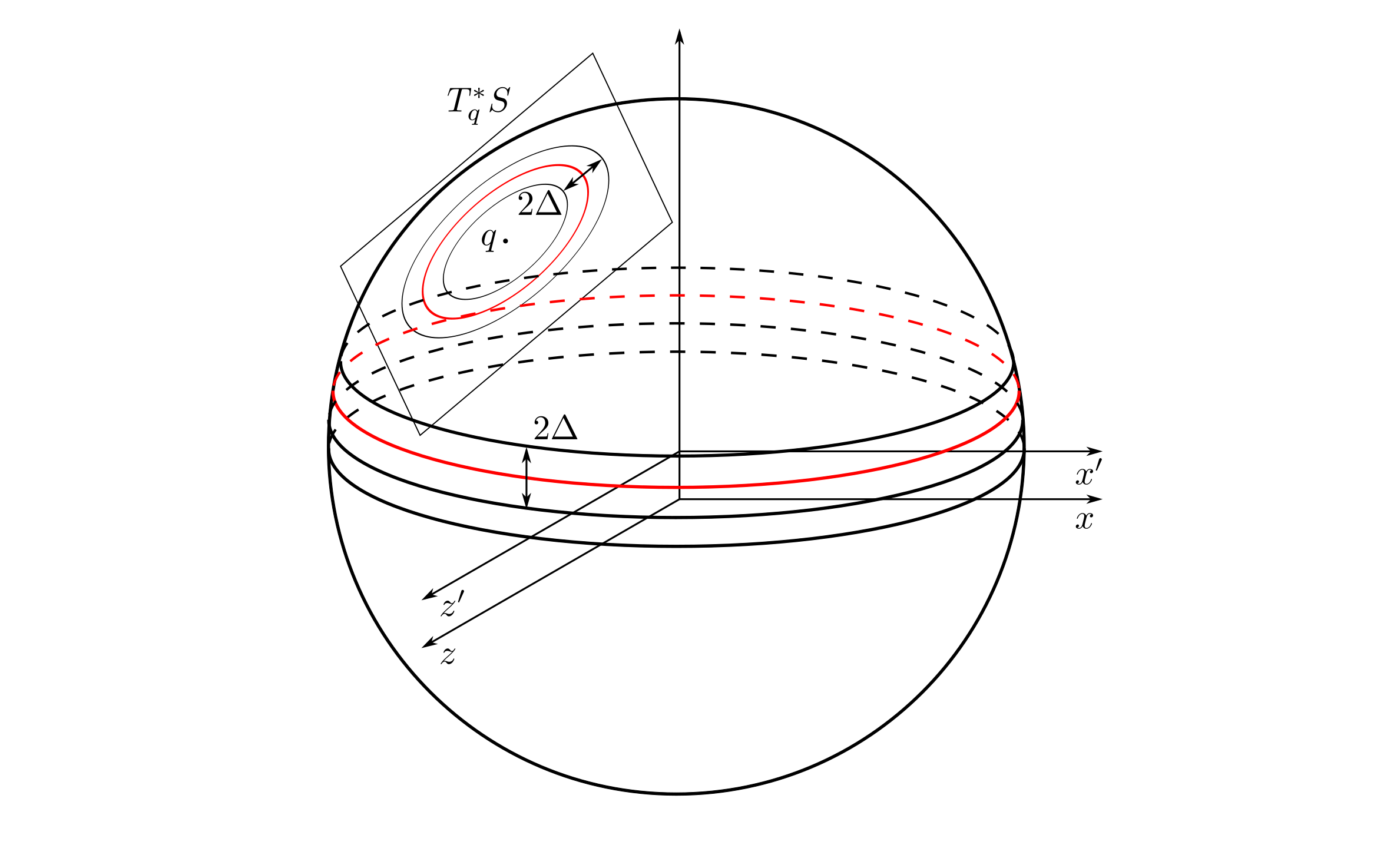}
  \caption{Regions where the original system is modified.}
  \label{fig:sfig2}
\end{figure}

Below we will also consider the following \textit{modified} system (Fig.~2)
\begin{align}
    \begin{split}
    \label{eq6}
        &\dot \theta = p_\theta +  \chi_{c,\delta}(\theta,\varphi,p_\theta,p_\varphi)\dot h \sin\theta \sin\varphi,\\
        &\dot \varphi = (p_\varphi - \chi_{c,\delta}(\theta,\varphi,p_\theta,p_\varphi)\dot h \cos\theta\cos\varphi)/\cos^2\theta,\\
        &\dot p_\theta = - (p_\varphi^2 - 2 \chi_{c,\delta}(\theta,\varphi,p_\theta,p_\varphi) p_\varphi \dot h \cos\varphi\cos\theta + \dot h^2 \cos^2\varphi\cos^2\theta) \sin\theta/\cos^3\theta\\
        &- \mu\dot\theta + \sin\theta\sin\varphi- P_x \sin\theta\cos\varphi + P_z \cos\theta - \dot h^2 \sin\theta\cos\theta\sin^2\varphi + \dot h^2 \sin\theta\cos^2\varphi/\cos\theta\\
        &- \chi_{c,\delta}(\theta,\varphi,p_\theta,p_\varphi)\dot h p_\theta \cos\theta\sin\varphi - \chi_{c,\delta}(\theta,\varphi,p_\theta,p_\varphi)\dot h p_\varphi \sin\theta\cos\varphi/\cos^2\theta\\
        &\dot p_\varphi = - \cos\theta\cos\varphi - \mu \dot\varphi\cos^2\theta - P_x \cos\theta \sin\varphi-\dot h^2 \sin^2\theta\cos\varphi\sin\varphi + \dot h^2 \cos\varphi\sin\varphi\\
        &-\chi_{c,\delta}(\theta,\varphi,p_\theta,p_\varphi)\dot h p_\theta \sin\theta\cos\varphi - \chi_{c,\delta}(\theta,\varphi,p_\theta,p_\varphi)\dot hp_\varphi \sin\varphi/\cos\theta.
    \end{split}
\end{align}
Here $\chi_{c,\delta}(\theta,\varphi,p_\theta,p_\varphi) = \sigma_\delta(\theta,\varphi) \cdot \rho_c(\theta,p_\theta,p_\varphi)$ and functions $\sigma_\delta(\theta,\varphi)$ and $\rho_c(\theta,p_\theta,p_\varphi)$ are defined as follows. Let $\Delta > 0$ be a small real number. Then we choose $\sigma_\delta(\theta,\varphi)$ to be a smooth function such that $\sigma_\delta = 0$ in the $\Delta$-neighborhood of the set $\cos\theta\sin\varphi = \delta$ and $\sigma_\delta = 1$ for all $\theta$ and $\varphi$ that do not lie in the $2\Delta$-neighborhood of this set. Similarly, the smooth function $\rho_c(\theta,p_\theta,p_\varphi)$ is defined as follows. For a given $\theta$, $\rho_c(\theta,p_\theta,p_\varphi) = 0$ for all $p_\theta$, $p_\varphi$ from the $\Delta$-neighborhood of the set $(p_\theta^2 + p_\varphi^2/\cos^2\theta)/2 = c^2$ and $\rho_c(\theta,p_\theta,p_\varphi) = 1$ for all $p_\theta$, $p_\varphi$ outside the $2\Delta$-neighborhood. These functions can be constructed explicitly but we omit this for brevity. If $\sigma_\delta \equiv 1$ and $\rho_c \equiv 1$, then systems \eqref{eq5} and \eqref{eq6} coincide.

\begin{remark}
Below we will assume that there are two one-parameter families of functions $\sigma_\delta(\theta,\varphi)$ and $\rho_c(\theta,p_\theta,p_\varphi)$, i.e. for any $\Delta > 0$ we have two smooth functions $\sigma_\delta(\theta,\varphi)$, $\rho_c(\theta,p_\theta,p_\varphi)$ satisfying the above properties. The exact form of these families needs not to be specified for the further considerations.
\end{remark}
All horizontal positions of the rod correspond to the following values of $\theta$ and $\varphi$
$$
\cos\theta\sin\varphi = 0.
$$


Therefore, when the rod is slightly above the horizon, we have
\begin{equation}
\label{eq7}
    \cos\theta\sin\varphi = \delta > 0,
\end{equation}
where $\delta$ is a relatively small real number. We will prove that any solution of the modified system that is tangent to the region defined by inequality \eqref{eq7} locally leaves this region, provided that $\delta$ is small. To be more precise, Lemma~1 holds (see Section 5). Below we will also consider the averaged version of system \eqref{eq6}:
\begin{align}
    \begin{split}
    \label{eq8}
        &\dot \theta = p_\theta,\\
        &\dot \varphi = p_\varphi/\cos^2\theta,\\
        &\dot p_\theta = - (p_\varphi^2 + \frac{1}{2} a^2 \left(\frac{2\pi}{T}\right)^2 \cos^2\varphi\cos^2\theta) \sin\theta/\cos^3\theta - \mu p_\theta + \sin\theta\sin\varphi- P_x(s) \sin\theta\cos\varphi +\\
        &P_z(s) \cos\theta - \frac{1}{2} a^2 \left(\frac{2\pi}{T}\right)^2 \sin\theta\cos\theta\sin^2\varphi + \frac{1}{2} a^2 \left(\frac{2\pi}{T}\right)^2 \sin\theta\cos^2\varphi/\cos\theta,\\
        &\dot p_\varphi = - \cos\theta\cos\varphi - \mu p_\varphi - P_x(s) \cos\theta \sin\varphi- \frac{1}{2} a^2 \left(\frac{2\pi}{T}\right)^2 \sin^2\theta\cos\varphi\sin\varphi + \frac{1}{2} a^2 \left(\frac{2\pi}{T}\right)^2 \cos\varphi\sin\varphi,\\
        &\dot s = 1.
    \end{split}
\end{align}

Here $s$ is an artificial time-like parameter which is introduced in order to distinguish two types of time-dependent functions. Functions $P_x$ and $P_z$ are not rapidly oscillating and we do not consider their averages. Function $\dot h$ can be averaged and we have
$$
\frac{1}{T} \int_0^{T} a \left(\frac{2\pi}{T}\right) \cos \left(\frac{2\pi}{T} t\right) \, dt = 0, \quad \frac{1}{T}\int_0^{T} a^2 \left(\frac{2\pi}{T}\right)^2 \cos^2 \left(\frac{2\pi}{T} t\right) \, dt = \frac{1}{2} a^2 \left(\frac{2\pi}{T}\right)^2.
$$
In other words, in the averaged system we put $\dot h$ to be zero and replace $\dot h^2$ with $\frac{1}{2} a^2 \left(\frac{2\pi}{T}\right)^2$. Lemmas 1 and 2 still hold for system \eqref{eq8} (see Lemmas 3 and 4).

\section{Main Result}

The general scheme of the proof can be described as follows. We will show that there exist $\delta > 0$ and $c > 0$ such that for any values $\varepsilon = 1/k$, $k \in \mathbb{N}$, $\Delta > 0$ there is a $T$-periodic of the modified system \eqref{eq6} and this solution never leaves set $M_{c,\delta}$, which is defined by the following two inequalities
\begin{equation}
\label{eq9}
\cos\theta\sin\varphi \geqslant \delta, \quad \frac{1}{2} \left( p_\theta^2 + \frac{p_\varphi^2}{\cos^2\theta} \right) \leqslant c^2.
\end{equation}
Then we will prove that this periodic solution cannot belong to the $2\Delta$-neighborhood of the boundary of $M_{c,\delta}$, provided that the absolute values of $\Delta$ and $\varepsilon$ are small. This follows from the results of the standard method of averaging: solutions of the modified \eqref{eq6} and the averaged \eqref{eq8} systems are close on a finite time interval when the frequency of oscillations of the pivot point is high (equivalently, $\varepsilon$ is small). Moreover, it can be shown that the solutions of the averaged system starting in a neighborhood of $\partial M_{c,\delta}$ leaves $M_{c,\delta}$ in  forward or reversed time. Hence, if the obtained periodic solution is close to the boundary, then we can conclude that this solution also leaves $M_{c,\delta}$. This contradiction proves that the periodic solution never belongs to the region where the original system is modified, i.e. this solution is also a solution for \eqref{eq5}.


\begin{theorem}
Given any $\mu, a, T > 0$ and any two $T$-periodic smooth functions $P_x$, $P_z$. There exists $\varepsilon = 1/k$, $k \in \mathbb{N}$ such that there is a $T$-periodic solution $(\theta(t), \varphi(t), p_\theta(t), p_\varphi(t))$ of system \eqref{eq5} and $\cos\theta(t)\sin\varphi(t) > 0$ for all $t$.
\end{theorem}

\begin{proof}
Let us consider system \eqref{eq6}. First, we specify parameters $c$ and $\delta$ in this modified system. Let $c$ be the constant from Lemma~2 and $c > 0$ be determined by Lemma~1. Let $M_{c,\delta}$ be a subset of the phase space defined by inequalities \eqref{eq9}. Let $W_{c,\delta}$ be the subset of the extended phase space such that
$$
W_{c,\delta} = M_{c,\delta} \times \mathbb{R}.
$$
We say that point $(\theta, \varphi, p_\theta, p_\varphi, t_0) \in \partial W_{c,\delta}$ is an egress point for $W_{c,\delta}$ if for some small $\epsilon > 0$ we have $(\theta(t_0 + t), \varphi(t_0 + t), p_\theta(t_0 + t), p_\varphi(t_0 + t), t_0 + t) \notin W_{c, \delta}$ for all $t \in (0, \epsilon)$; here $(\theta(t_0 + t), \varphi(t_0 + t), p_\theta(t_0 + t), p_\varphi(t_0 + t))$ is the solution of system \eqref{eq6} such that $(\theta(t_0), \varphi(t_0), p_\theta(t_0), p_\varphi(t_0)) = (\theta, \varphi, p_\theta, p_\varphi)$.

Let us denote the set of egress points by $W^{-}_{c,\delta}$. From Lemma~2 we obtain that point $(\theta, \varphi, p_\theta, p_\varphi, t_0) \notin W^{-}_{c,\delta}$ if
$$
\frac{1}{2} \left( p_\theta^2 + \frac{p_\varphi^2}{\cos^2\theta} \right) = c^2,
$$
and
$$
\cos\theta\sin\varphi > \delta.
$$
Therefore, for all points of $W^{-}_{c,\delta}$ we have
$$
\cos\theta\sin\varphi = \delta.
$$
From Lemma~1 we obtain that points satisfying
$$
\left.\frac{d}{dt}(\cos\theta(t)\sin\varphi(t))\right|_{t=t_0}=-p_\theta \sin\theta\sin\varphi + \frac{p_\varphi}{\cos\theta}\cos\varphi = 0.
$$
are in $W^{-}_{c,\delta}$. Clearly, point $(\theta, \varphi, p_\theta, p_\varphi, t_0) \in W^{-}_{c,\delta}$ when
$$
\left.\frac{d}{dt}(\cos\theta(t)\sin\varphi(t))\right|_{t=t_0}=-p_\theta \sin\theta\sin\varphi + \frac{p_\varphi}{\cos\theta}\cos\varphi < 0,
$$
and $(\theta, \varphi, p_\theta, p_\varphi, t_0) \notin W^{-}_{c,\delta}$ if
$$
\left.\frac{d}{dt}(\cos\theta(t)\sin\varphi(t))\right|_{t=t_0}=-p_\theta \sin\theta\sin\varphi + \frac{p_\varphi}{\cos\theta}\cos\varphi > 0.
$$


Therefore, $W^{-}_{c,\delta}$ is compact and homotopic to a circle.  $W_{c, \delta}$ is also compact and homotopic to a point. In particular, $\chi(W_{c, \delta}) - \chi(W^-_{c, \delta}) = 1$ (here $\chi$ denotes the usual Euler characteristic). From the result by R. Srzednicki \cite{srzednicki1994periodic} (see also \cite{srzednicki2005fixed} or \cite{polekhin2015forced}), we obtain that for any $\varepsilon > 0$ and $\Delta > 0$ there exists a $T$-periodic solution $(\theta(t), \varphi(t), p_\theta(t), p_\varphi(t))$ of system \eqref{eq6} such that $(\theta(t), \varphi(t), p_\theta(t), p_\varphi(t), t) \in W_{c,\delta}$ for all $t$.

Let $d$ be the constant from Corollary 1 and $2\Delta < d$. Then there exists a $T$-periodic solution for any $\varepsilon > 0$. We will show that for small $\varepsilon > 0$ this solution cannot belong to the $2\Delta$-neighborhood of $M_{c, \delta}$. In other words, the trajectory of this solutions does not run through the set where we modify the original system.

Let $\varepsilon_0 > 0$ be a number such that for any $\varepsilon \in (0, \varepsilon_0)$ and for any initial condition from the $2\Delta$-neighborhood of $\partial M_{c,\delta}$, the corresponding solutions of the modified and the averaged systems are $l/2$-close for $t \in [t_0 - \tau, t_0 + \tau]$. Here $l$ and $\tau$ are the constants from Corollary 1. The existence of such $\varepsilon_0$ follows from  a classical theorem on averaging on a finite time interval for compacts \cite{sanders2007averaging}.

Suppose that for some $t_0$, the periodic solution of the modified system belongs to the $2\Delta$-neighborhood of $\partial M_{c,\delta}$. From the corollary we obtain that $(\theta(t), \varphi(t), p_\theta(t), p_\varphi(t), t) \notin W_{c, \delta}$ for some $t \in [t_0 - \tau, t_0 + \tau]$ (Fig.~3).

\begin{figure}[h!]
  \centering
  \includegraphics[width=0.66\linewidth]{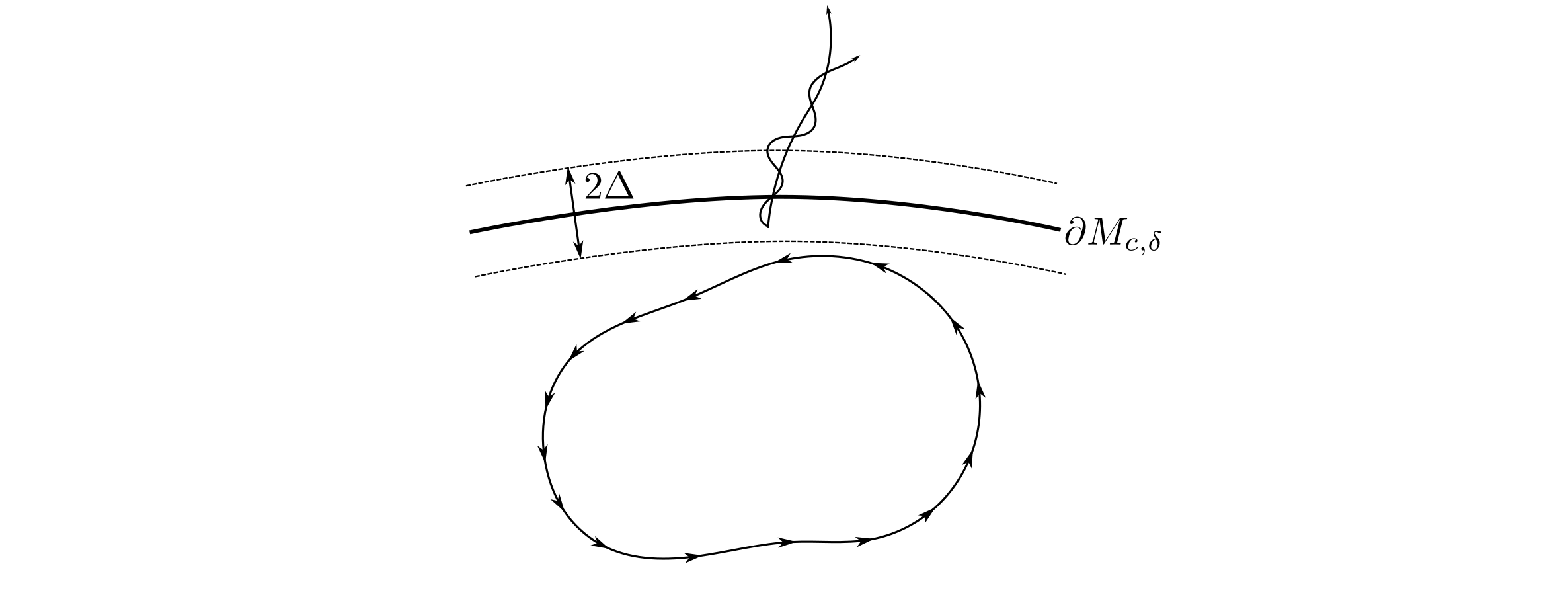}
  \caption{The periodic solution cannot be close to the boundary.}
  \label{fig:sfig2}
\end{figure}

It contradicts the property of the considered periodic solution to remain in $W_{c, \delta}$ for all $t$. Finally, we conclude that the periodic solution obtained for the modified system remains a solution for the original system.
\end{proof}

Below we present some numerical results concerning periodic solutions of system \eqref{eq5}. Our main goal is to show that periodic solutions without falling, the existence of which is guaranteed by Theorem 1, can be asymptotically stable. One can see that the obtained periodic trajectories may have relatively complex forms.
\begin{figure}[h!]
\begin{subfigure}{.33\textwidth}
  \centering
  \includegraphics[width=1.0\linewidth]{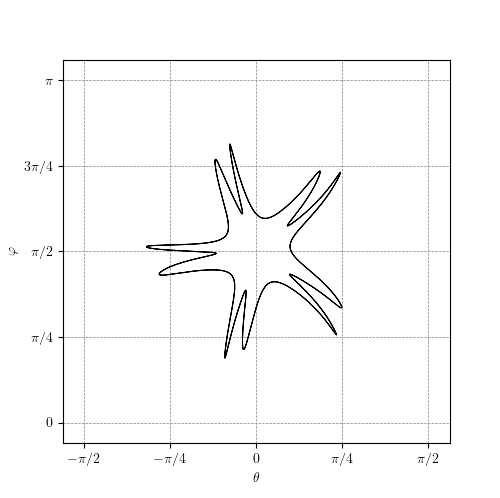}
  \caption{$k = 10$, $a = 5$, $A = 6$.}
  \label{fig:sfig2}
\end{subfigure}%
\begin{subfigure}{.33\textwidth}
  \centering
  \includegraphics[width=1.0\linewidth]{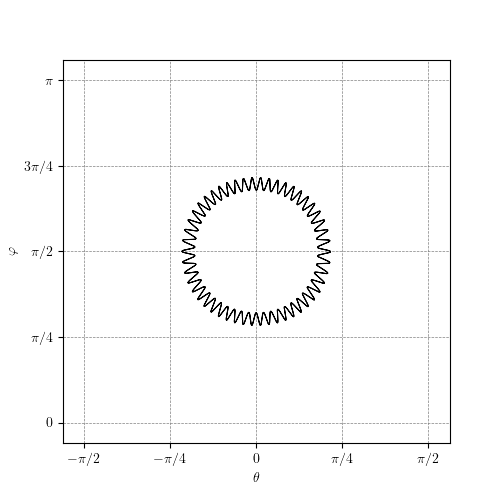}
  \caption{$k = 50$, $a = 5$, $A = 6$.}
  \label{fig:sfig1}
\end{subfigure}
\begin{subfigure}{.33\textwidth}
  \centering
  \includegraphics[width=1.0\linewidth]{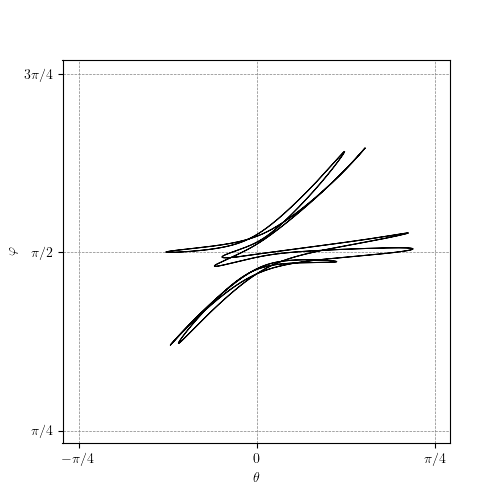}
  \caption{$k = 10$, $a = 5$, $A = 3/2$, $\alpha= \pi/2$.}
  \label{fig:sfig2}
\end{subfigure}\\
\begin{subfigure}{.33\textwidth}
  \centering
  \includegraphics[width=1.0\linewidth]{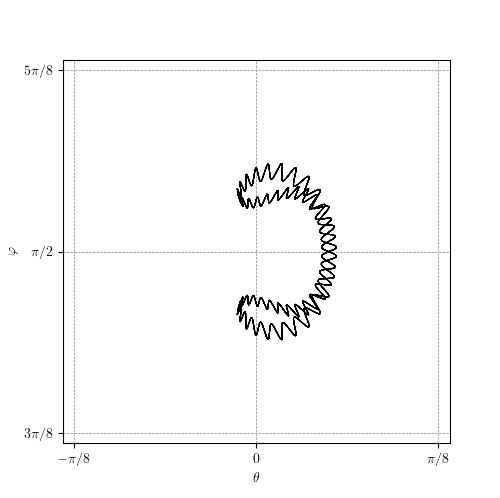}
  \caption{$k = 50$, $a = 5$, $A = 3/2$, $\alpha= \pi/2$.}
  \label{fig:sfig1}
\end{subfigure}%
\begin{subfigure}{.33\textwidth}
  \centering
  \includegraphics[width=1.0\linewidth]{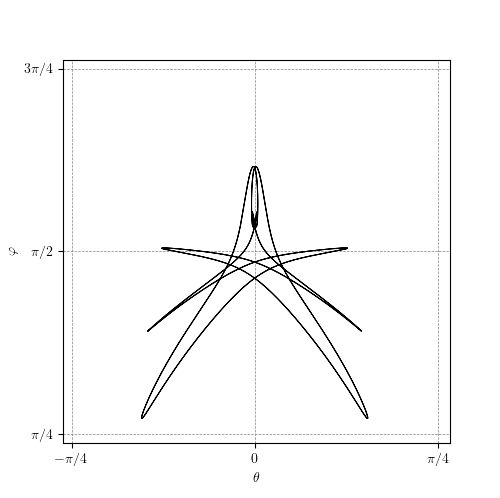}
  \caption{$k = 10$, $a = 5$, $A = 3/2$, $\alpha= \pi$.}
  \label{fig:sfig2}
\end{subfigure}%
\begin{subfigure}{.33\textwidth}
  \centering
  \includegraphics[width=1.0\linewidth]{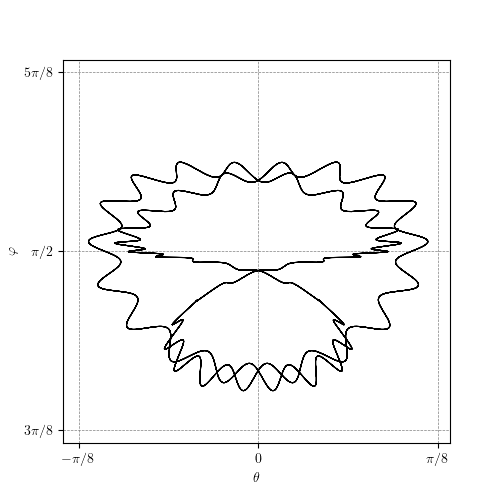}
  \caption{$k = 50$, $a = 5$, $A = 3/2$, $\alpha= \pi$.}
  \label{fig:sfig1}
\end{subfigure}\\
\label{figEarth}
\caption{Asymptotically stable $2\pi$-periodic solutions for $P_x(t) = A \cos t$, $P_z(t) = A \sin t$ and for $P^\alpha_x(t) = A \cos (\alpha - \alpha\cos t)$, $P^\alpha_z(t) = A \sin(\alpha - \alpha\cos t)$.}
\end{figure}

System \eqref{eq5} is determined by the following list of parameters: $a$, $\varepsilon$, $T$, $\mu$. We will consider the cases when $T = 2\pi$, $\mu = 1$, and parameters $a$ and $\varepsilon$ vary. We will also consider two types of functions $P_x(t)$, $P_z(t)$, namely:
\begin{enumerate}
    \item $P_x(t) = A \cos t$, $P_z(t) = A \sin t$,
    \item $P^\alpha_x(t) = A \cos (\alpha - \alpha\cos t)$, $P^\alpha_z(t) = A \sin(\alpha - \alpha\cos t)$.
\end{enumerate}

Here $\alpha, A \in \mathbb{R}$. To a large degree this choice of functions $P_x(t)$, $P_z(t)$ is voluntary and many other types of the external fields can be considered. In our case, the first type defines a rotating force field of magnitude $A$, and the second type is a vector field oscillating between the angles zero and $2\alpha$; the magnitude in this case also equals $A$.


\section{Technical Lemmas}

\begin{lemma}
Given any positive $c > 0$. For system \eqref{eq6} there exists $\delta > 0$ such that for any $\Delta > 0$ and for any initial conditions $\theta = \theta(t_0)$, $\varphi = \varphi(t_0)$, $p_{\theta} = p_{\theta}(t_0)$, and $p_{\varphi} = p_{\varphi}(t_0)$ satisfying $\left.\cos\theta(t)\sin\varphi(t)\right|_{t=t_0} = \delta$, $\left.\frac{d}{dt}(\cos\theta(t)\sin\varphi(t))\right|_{t=t_0} = 0$, and
$\left.\frac{1}{2} \left( p_\theta^2(t) + \frac{p_\varphi^2(t)}{\cos^2\theta(t)} \right)\right|_{t=t_0}\leqslant c^2$, and for any $t_0$ we have $\left.\frac{d^2}{dt^2}(\cos\theta(t)\sin\varphi(t))\right|_{t=t_0} < 0 $.
\end{lemma}
\begin{proof}
The proof is a direct calculation. First, from \eqref{eq5}, we obtain
\begin{equation*}
    -p_\theta \sin\theta\sin\varphi + \frac{p_\varphi}{\cos\theta}\cos\varphi = 0.
\end{equation*}
Here we use that $\sigma_\delta(\theta,\varphi) = 0$ for the considered points. In particular, $\dot \theta = p_\theta$ and $\dot\varphi = p_\varphi/\cos^2\theta$. For the second derivative we have the expression
\begin{align*}
    \left.\frac{d^2}{dt^2}(\cos\theta\sin\varphi)\right|_{t=t_0}= -\dot p_\theta \sin\theta \sin\varphi + \frac{\dot p_\varphi}{\cos\theta}\cos\varphi - p_\theta^2 \cos\theta\sin\varphi - p_\varphi^2\frac{\sin\varphi}{\cos^3\theta}.
\end{align*}
Finally, for small $\delta > 0$
\begin{align*}
    &\left.\frac{d^2}{dt^2}(\cos\theta\sin\varphi)\right|_{t=t_0} = -1 + \cos^2\theta\sin^2\varphi\\ &-(-(p_\varphi^2 + \dot h^2 \cos^2\varphi \cos^2\theta) \sin\theta / \cos^3\theta- \mu\dot\theta + \sin\theta\sin\varphi - P_x \sin\theta\cos\varphi + P_z \cos\theta\\
    &- \dot h^2 \sin\theta\cos\theta\sin^2\varphi + \dot h^2 \sin\theta\cos^2\varphi/\cos\theta ) \sin\theta \sin\varphi - p_\theta^2 \cos\theta\sin\varphi - p_\varphi^2\frac{\sin\varphi}{\cos^3\theta}  \\
    &+ (- \cos\theta\cos\varphi - \mu \dot\varphi\cos^2\theta - P_x \cos\theta \sin\varphi-\dot h^2 \sin^2\theta\cos\varphi\sin\varphi + \dot h^2 \cos\varphi\sin\varphi)\frac{\cos\varphi}{\cos\theta}\\
    & = P_x \cos^2\theta\cos\varphi\sin\varphi - P_z \cos\theta\sin\theta\sin\varphi + \mu(p_\theta \sin\theta\sin \varphi - p_\varphi \frac{\cos\varphi}{\cos\theta})\\
    &+ \dot h^2 (\cos^2\varphi\sin\varphi/\cos\theta   + \sin^2\theta\cos\theta\sin^3\varphi - \sin^2\theta\cos^2\varphi\sin\varphi/\cos\theta)\\
    &+ p_\varphi^2(\sin^2\theta \sin\varphi/ \cos^3\theta - \sin\varphi/ \cos^3\theta) - p_\theta^2 \cos\theta\sin\varphi = -1 + C\delta < 0.
\end{align*}
Here the constant $C$ depends on $P_x$, $P_z$, $c$, $\dot h^2$. The latter inequality holds for any $0 < \delta < 1/C$.
\end{proof}

\begin{lemma}
There exists $c$ such that for system \eqref{eq6} for any $\Delta > 0$ and for any initial conditions $\theta = \theta(t_0)$, $\varphi = \varphi(t_0)$, $p_{\theta} = p_{\theta}(t_0)$, and $p_{\varphi} = p_{\varphi}(t_0)$ satisfying $$\frac{1}{2} \left( p_\theta^2 + \frac{p_\varphi^2}{\cos^2\theta} \right) = c^2 \mbox{ and } \cos\theta \ne 0$$ we have $$\frac{d}{dt}\left.\frac{1}{2} \left( p_\theta^2(t) + \frac{p_\varphi^2(t)}{\cos^2\theta(t)} \right)\right|_{t=t_0} < 0.$$
\end{lemma}
\begin{proof}
Let us calculate the derivative
\begin{align*}
    \frac{d}{dt}\left.\frac{1}{2} \left( p_\theta^2(t) + \frac{p_\varphi^2(t)}{\cos^2\theta(t)} \right)\right|_{t=t_0} = p_\theta \dot p_\theta + \frac{p_\varphi \dot p_\varphi}{\cos^2\theta} + p_\varphi^2 \frac{\sin\theta \dot\theta}{\cos^3\theta}.
\end{align*}
Since $\rho_c(\theta,p_\theta,p_\varphi) = 0$ in all considered points, we can rewrite this expression as follows
\begin{align*}
    &p_\theta \dot p_\theta + \frac{p_\varphi \dot p_\varphi}{\cos^2\theta} + p_\varphi \frac{\sin\theta p_\theta}{\cos^3\theta} = \\
    &p_\theta (- (p_\varphi^2 + \dot h^2 \cos^2\varphi\cos^2\theta) \sin\theta/\cos^3\theta - \mu p_\theta + \sin\theta\sin\varphi- P_x \sin\theta\cos\varphi +\\
    &P_z \cos\theta - \dot h^2 \sin\theta\cos\theta\sin^2\varphi + \dot h^2 \sin\theta\cos^2\varphi/\cos\theta) + p_\varphi^2 p_\theta \sin\theta  /\cos^3\theta +\\
    & p_\varphi (- \cos\theta\cos\varphi - \mu p_\varphi - P_x \cos\theta \sin\varphi-\dot h^2 \sin^2\theta\cos\varphi\sin\varphi + \dot h^2 \cos\varphi\sin\varphi)/\cos^2\theta=\\
    & -2\mu c^2 + p_\theta (\sin\theta\sin\varphi - P_x \sin\theta\cos\varphi + P_z\cos\theta - \dot h^2 \sin\theta\cos\theta\sin^2\varphi) +\\
    &p_\varphi / \cos\theta \cdot (-\cos\varphi - P_x\sin\varphi + \dot h^2 \cos\varphi \sin\varphi \cos\theta).
\end{align*}
Since $p_\theta \in [-c,c]$ and $p_\varphi/\cos\theta \in [-c,c]$, then the sign of the expression is defined by the first term ($-2\mu c^2$), provided that $c$ is large.
\end{proof}
\begin{lemma}
Given any positive $c > 0$. For system \eqref{eq8} there exists $\delta > 0$ such that for system \eqref{eq8} for any initial conditions $\theta$, $\varphi$, $p_{\theta}$, $p_{\varphi}$ and $s_0$ satisfying $\left.\cos\theta(t)\sin\varphi(t)\right|_{t=0} = \delta$, $\left.\frac{d}{dt}(\cos\theta(t)\sin\varphi(t))\right|_{t=0} = 0$, and
$\left.\frac{1}{2} \left( p_\theta^2(t) + \frac{p_\varphi^2(t)}{\cos^2\theta(t)} \right)\right|_{t=0}\leqslant c^2$, we have $\left.\frac{d^2}{dt^2}(\cos\theta(t)\sin\varphi(t))\right|_{t=0} < 0 $.
\end{lemma}
\begin{lemma}
There exists $c$ such that for system \eqref{eq8} for any initial conditions $\theta$, $\varphi$, $p_{\theta}$, $p_{\varphi}$, and $s_0$ satisfying $$\frac{1}{2} \left( p_\theta^2 + \frac{p_\varphi^2}{\cos^2\theta} \right) = c^2 \mbox{ and } \cos\theta \ne 0$$ we have $$\frac{d}{dt}\left.\frac{1}{2} \left( p_\theta^2 + \frac{p_\varphi^2}{\cos^2\theta} \right)\right|_{t=0} < 0.$$
\end{lemma}
Let $M_{c,\delta}$ be the subset of the phase space such that
$$
\cos\theta\sin\varphi \geqslant \delta, \quad \frac{1}{2} \left( p_\theta^2 + \frac{p_\varphi^2}{\cos^2\theta} \right) \leqslant c^2.
  $$
As a corollary from Lemmas 3 and 4 we prove the following result.
\begin{corollary}
For system \eqref{eq8} there exist positive $c$, $\delta$, $d$, $\tau$, $l$ such that for any $(\theta, \varphi, p_\theta, p_\varphi) \in M_{c, \delta}$ satisfying $\mathrm{dist}((\theta, \varphi, p_\theta, p_\varphi), \partial M_{c, \delta}) \leqslant d$ and any $s_0$ we have for some $t \in [0, \tau]$ either
$$
(\theta(t), \varphi(t), p_\theta(t), p_\varphi(t)) \not\in M_{c, \delta}, \quad \mathrm{dist}((\theta(t), \varphi(t), p_\theta(t), p_\varphi(t)), \partial M_{c, \delta}) > l,
$$
or
$$
(\theta(-t), \varphi(-t), p_\theta(-t), p_\varphi(-t)) \not\in M_{c, \delta}, \quad \mathrm{dist}((\theta(-t), \varphi(-t), p_\theta(-t), p_\varphi(-t)), \partial M_{c, \delta}) > l.
$$
Here $\theta(0)=\theta$, $\varphi(0)=\varphi$, $p_\theta(0) = p_\theta$, $p_\varphi(0)=p_\varphi$, and $s(0) = s_0$.
\end{corollary}
\begin{proof}
Let $c$ and $\delta$ be constants such that Lemmas 3 and 4 hold and let $\partial M_{c, \delta}^1$ and $\partial M_{c, \delta}^2$ be the following components of the boundary $\partial M_{c, \delta} = \partial M_{c, \delta}^1 \cup \partial M_{c, \delta}^2$
\begin{align*}
   &\partial M_{c, \delta}^1 = \{ (\theta, \varphi, p_\theta, p_\varphi) \in \partial M_{c, \delta} \colon \frac{1}{2} \left( p_\theta^2 + \frac{p_\varphi^2}{\cos^2\theta} \right) = c^2 \},\\
   &\partial M_{c, \delta}^2 = \{ (\theta, \varphi, p_\theta, p_\varphi) \in \partial M_{c, \delta} \colon \cos\theta\sin\varphi = \delta,  \frac{1}{2} \left( p_\theta^2 + \frac{p_\varphi^2}{\cos^2\theta} \right) \leqslant c^2 \}.
\end{align*}
From the compactness of $\partial M_{c, \delta}^1$ and Lemma 4, we obtain that there exist $d$, $\tau$, $l$ such that for any $s_0$ and any initial conditions $\theta(0)=\theta$, $\varphi(0)=\varphi$, $p_\theta(0) = p_\theta$, $\mathrm{dist}((\theta, \varphi, p_\theta, p_\varphi), \partial M_{c, \delta}^1) \leqslant d$, we have for some $t \in [0, \tau]$
 $$
(\theta(-t), \varphi(-t), p_\theta(-t), p_\varphi(-t)) \not\in M_{c, \delta}, \quad \mathrm{dist}((\theta(-t), \varphi(-t), p_\theta(-t), p_\varphi(-t)), \partial M_{c, \delta}) > l.
$$
Let us now consider set $\partial M_{c, \delta}^2$. We can find $t_1$ and $l_1$ such that for any initial conditions satisfying
\begin{align}
\label{eq9}
    \left.\frac{d}{dt}(\cos\theta(t)\sin\varphi(t))\right|_{t=0} = 0
\end{align}
we have
 $$
(\theta(t_1), \varphi(t_1), p_\theta(t_1), p_\varphi(t_1)) \not\in M_{c, \delta}, \quad \mathrm{dist}((\theta(t_1), \varphi(t_1), p_\theta(t_1), p_\varphi(t_1)), \partial M_{c, \delta}) > l_1.
$$
It follows from Lemma 3. Therefore, the same is true for all close initial conditions from $M_{c,\delta}$. We will denote by $\partial M_{c, \delta}^{2,0}$ the set of initial conditions from $\partial M_{c, \delta}^2$ satisfying \eqref{eq9}. Based on the compactness of the considered sets, we can conclude that there exist $d$, $\tau$, $l$ such that for any $s_0$ and any initial conditions $\theta(0)=\theta$, $\varphi(0)=\varphi$, $p_\theta(0) = p_\theta$, $\mathrm{dist}((\theta, \varphi, p_\theta, p_\varphi), \partial M_{c, \delta}^{2,0}) \leqslant d$, we have for some $t \in [0, \tau]$
 $$
(\theta(t), \varphi(t), p_\theta(t), p_\varphi(t)) \not\in M_{c, \delta}, \quad \mathrm{dist}((\theta(t), \varphi(t), p_\theta(t), p_\varphi(t)), \partial M_{c, \delta}) > l.
$$
By $\partial M_{c, \delta}^{2,+}$ we denote the set of points from $\partial M_{c, \delta}^2$ satisfying
\begin{align*}
    \left.\frac{d}{dt}(\cos\theta(t)\sin\varphi(t))\right|_{t=0} > 0.
\end{align*}
Again, there exist $d$, $\tau$, $l$ such that for any $s_0$ and any initial conditions $\theta(0)=\theta$, $\varphi(0)=\varphi$, $p_\theta(0) = p_\theta$ satisfying $\mathrm{dist}((\theta, \varphi, p_\theta, p_\varphi), \partial M_{c, \delta}^{2,0}) > d$ and $\mathrm{dist}((\theta, \varphi, p_\theta, p_\varphi), \partial M_{c, \delta}^{2,+}) \leqslant d$, we have for some $t \in [0, \tau]$
 $$
(\theta(-t), \varphi(-t), p_\theta(-t), p_\varphi(-t)) \not\in M_{c, \delta}, \quad \mathrm{dist}((\theta(-t), \varphi(-t), p_\theta(-t), p_\varphi(-t)), \partial M_{c, \delta}) > l.
$$
Similarly, if $\partial M_{c, \delta}^{2,-}$ is the subset of $\partial M_{c, \delta}^2$ satisfying
\begin{align*}
    \left.\frac{d}{dt}(\cos\theta(t)\sin\varphi(t))\right|_{t=0} < 0.
\end{align*}
Then there exist $d$, $\tau$, $l$ such that for any $s_0$ and any initial conditions $\theta(0)=\theta$, $\varphi(0)=\varphi$, $p_\theta(0) = p_\theta$ satisfying $\mathrm{dist}((\theta, \varphi, p_\theta, p_\varphi), \partial M_{c, \delta}^{2,0}) > d$ and $\mathrm{dist}((\theta, \varphi, p_\theta, p_\varphi), \partial M_{c, \delta}^{2,-}) \leqslant d$, we have for some $t \in [0, \tau]$
 $$
(\theta(t), \varphi(t), p_\theta(t), p_\varphi(t)) \not\in M_{c, \delta}, \quad \mathrm{dist}((\theta(t), \varphi(t), p_\theta(t), p_\varphi(t)), \partial M_{c, \delta}) > l.
$$
\end{proof}

\section{Conclusion}

In conclusion we would like to discuss briefly possible generalizations of the presented result and some open questions. First, a natural question is whether Theorem 1 can be carried over to systems without friction. The answer is positive and the general idea of the proof remains the same, but it is slightly more  complicated technically. Similarly, the result can be generalized to the case when the spherical pendulum is replaced by an arbitrary surface and the mass point is constrained to move on this surface. One can also note that the proof of Theorem still holds (with minor changes) if we replace the homogeneous horizontal force field with a non-homogeneous one, i.e. if we replace functions $P_x(t)$ and $P_z(t)$ with $P_x(t,x,y,z)$ and $P_z(t,x,y,z)$.

More sophisticated questions appear when we study Lyapunov stability of periodic solutions in the considered systems. If sufficient conditions for the existence of an asymptotically stable $T$-periodic solution are needed, then this problem can be reduced to the consideration of the same question for the averaged system: if there exists a $T$-periodic non-falling solution of the averaged system, and all multipliers of this solution are strictly inside the unit circle, then there exists a $T$-periodic solution with the same properties for the original system, provided that $\varepsilon$ is small. If $P(t) \equiv 0$, then one can easily find conditions under which the vertical equilibrium becomes asymptotically stable. It would be interesting to find similar conditions for the case when $P(t) \not\equiv 0$. For instance, one can  consider a relatively simple expressions defining the horizontal force  $P_x(t) = A \cos t$, $P_z(t) = A \sin t$ and try to find a non-trivial correspondence between $A$ and $a$ which guarantees the existence of an asymptotically stable solution. This type of results will complete the generalization of the spherical Kapitza pendulum. We have already proved that there always exists a periodic non-falling solution in the presence of a periodic horizontal force (in case of the Kapitza pendulum, this solution also always exists and it is the vertical equilibrium). The conditions for the asymptotic stability for the Kapitza pendulum are known and it might be useful to find similar conditions for system \eqref{eq5}. Another open problem related to the considered systems is the problem of precise estimation of values of $\varepsilon$ for which there exist a periodic solution. In other words, it would be interesting to know the maximal $\varepsilon$ guaranteeing the existence of a non-falling $T$-periodic solution.
\printbibliography
\end{document}